\documentclass[12pt]{amsart}

\usepackage{amsmath,amssymb}
\usepackage{hyperref}
\usepackage{datetime}
\usepackage{xcolor}
 \usepackage{float}
 
\newtheorem{thm}{Theorem}
\newtheorem{definition}{Definition}
\newtheorem{lem}[thm]{Lemma}
\newtheorem{prop}[thm]{Proposition}
\newtheorem{cor}[thm]{Corollary}

\newcommand{\Z}{{\mathbb Z}} 
 
\newcommand{\R}{{\mathbb R}}
\newcommand{\C}{{\mathbb C}} 
\newcommand{\corrRIk}{{\operatorname{Corr}(R,k,I)}}
\newcommand{\corrRk}{{\operatorname{Corr}(R,k)}} 
\newcommand{\mean}{{\mathrm{Mean}}}

\newcommand{\fixmehide}[1]{{}}

\usepackage{graphicx}
\graphicspath{ {./images/} }

\title[Area correlations]{Area
  correlations related to lattice points in discs}
\author{Matteo Bordignon}

\address{Dipartimento di Matematica “F. Enriques”, Università degli
  Studi di Milano, Via Saldini 50, 20133 Milano, Italy.} 
\email{matteo.bordignon@unimi.it}

\author{P\"ar Kurlberg}
\address{Department of Mathematics, KTH,
SE-100 44 Stockholm, Sweden}
\email{kurlberg@kth.se}

\date{\today}

\begin{document}
\begin{abstract}
  Motivated by the Lester-Wigman vanishing area correlation conjecture
  for lattice points near the boundary of circles of growing radius,
  we investigate the dynamics of circle flows on tori (which are
  related to the motion of a charged particle in a magnetic field on a
  torus.)  We show that an analogue of the vanishing correlation
  conjecture holds in this setting, i.e., we have ``mixing for the
  area observable'' despite the flow being essentially integrable.  We
  also determine the probability density function of the areas, in
  global as well as local regimes.
  
\end{abstract}
\date{}
\maketitle


\section{Introduction}
Given $R > 0$, let $$B(R) := \{ x \in \R^{2} : |x| \le R \}$$ denote
the radius $R$ ball with center at $(0,0)$, let the circle $C(R)$
denote the boundary of $B(R)$,
and let $\Lambda := \Z^{2}$ denote the standard lattice in $\R^{2}$.  A
standing assumption is that $R$ is chosen so that $C(R) \cap \Lambda =
\emptyset$; it suffices to avoid a countable subset of radii for this
to hold.
The
``half open'' unit square $S := [0,1) \times [0,1)$ is a fundamental
domain for the action of $\Lambda$ on $\R^2$. For $\lambda \in \Lambda$ define
$S_{\lambda} := S+\lambda$ and note the {\em disjoint} union
$\cup_{\lambda \in \Lambda} S_{\lambda} = \R^{2}$.

For $R > 0$ define
$$\Lambda(R) := \{ \lambda \in \Lambda : S_{\lambda} \cap C(R) \neq
\emptyset \}$$
and for $\lambda \in \Lambda(R)$, let $A(\lambda)$ denote the area of
$S_{\lambda} \cap B(R)$. The  main result of this paper is 
that the area correlation, for $k \in \Z$,
$$
\corrRk := \frac{1}{|\Lambda(R)|}
\sum_{\lambda_{j} \in \Lambda(R)}
(A(\lambda_{j})-1/2) \cdot (A(\lambda_{j+k}) -1/2)
$$
vanishes as $R \to \infty$ provided $|k| = O(R^{1-\epsilon})$ grows
with $R$, where $\{ \lambda_{1}, \lambda_{2}, \ldots \}$ denotes the
elements of $ \Lambda(R)$ ordered by increasing argument (after
identifying $\R^{2}$ with $\C$.)
We subtract by $1/2$ since the mean
area is $1/2 +o(1)$ as $R \to \infty$,
cf. Proposition~\ref{prop:mean-variance}. 
Further, here and in what follows we use the
convention of cyclic wraparound in case $j+k > |\Lambda(R)|$ or
$j + k < 1$ (in case $k<0$).

To prove that the (global) area correlation vanishes we work locally. Given an  interval 
$I = I(R) \subset [0,2 \pi]$, allowed to shrink as  $R$ grows, we
define 
$$
\Lambda(R,I) := \{ \lambda \in \Lambda(R) : \arg(\lambda) \in I \},
$$
and study the local area correlations
$$
\corrRIk :=
\frac{1}{|\Lambda(R,I)|}
\sum_{\lambda_{j} \in \Lambda(R,I)}
(A(\lambda_{j})-1/2) \cdot (A(\lambda_{j+k}) -1/2)
$$

In fact, when\footnote{By $k \asymp R^{1-\epsilon}$ we mean that
  $k \ll R^{1-\epsilon} \ll k$. } $k \asymp R^{1-\epsilon}$ we must
allow $|I|= o(1/k) = o(1/R^{1-\epsilon})$ for ``most'' such
short intervals $I$ --- a key issue is to avoid intervals for
which the slope of the tangent line of the circle is well approximated
by a rational number of low height.

\begin{thm}
  \label{thm:vanishing-local-correlations}
  Given $\epsilon > 0$, provided 
  $$k = k(R) = O(R^{1-\epsilon})$$
  and 
  $$|I| > (\log \log Q)/\log Q\,\, \mathrm{with}\,\,Q =
 \min(\log R, \sqrt{k}/\log k),$$
 we have
  $$\corrRIk = o(1),$$ 
  as $k,R \to \infty$.\fixmehide{remark that we can fix $k$ take $R \to
    \infty$ and then let $k$ grow? Existence of $C_{k}$ as in LW?}

\end{thm}
There is a tradeoff between taking $k$ large and $I$ having small
size.  Our primary interest is taking $k$ essentially as large as
possible, but the methods allow us to take, say,
$|I| \gg (\log R)^{8}/\sqrt{R}$ at the expense of having to further
restrict the size of $k$.  We note that there is a natural barrier near
$|I| = R^{-1/2}$: for $|I| = o(R^{-1/2})$ the local mean area can
significantly deviate from $1/2$ (e.g. at the lower part of the
circle.)


For comparison,  we also prove
some simpler statistics related to the areas $A(\lambda_{j})$.
We define the (global) mean area as  
$$\mean(R) := \frac{1}{|\Lambda(R)|}
\sum_{\lambda_{j} \in \Lambda(R)}
A(\lambda_{j}),$$
and the (global) variance as
$$\mathrm{Var}(R) := \frac{1}{|\Lambda(R)|}
\sum_{\lambda_{j} \in \Lambda(R)}
\left(A(\lambda_{j})-\mean(R) \right)^2.$$
We prove the following asymptotics for the global mean and
variance. 
\begin{prop}
  \label{prop:mean-variance}
    We have 
    $$\mean(R) =\frac{1}{2}+o(1),$$
    and 
    $$\mathrm{Var}(R)=
c_{\mathrm{glob}}+o(1), \quad c_{\mathrm{glob}} :=
\frac{13}{60}-\frac{1}{12}\log(1+\sqrt2)-\frac{1}{30\sqrt2} =
0.119\ldots >0, $$ 
    as $R \to \infty$.
\end{prop}
In fact, we compute the asymptotic for the local mean and variance
where we allow $|I|$ to shrink in two regimes.
Namely, for  $|I|>1/R^{\frac{1}{2}-\epsilon}$
(and any small fixed $\epsilon>0$)
as
well as for much shorter intervals at ``generic'' position; here
generic is related to Diophantine properties of
$\tan(\arg(\lambda))$), see Section \ref{sec:localmean}.  While the
expectation is locally constant\footnote{We remark that for
  $|I| = o(\sqrt{R})$ the local mean is generically $1/2+o(1)$, but
  for non-generic placements of $I$ the local mean might deviate from
  $1/2$.} for $|I|>1/R^{\frac{1}{2}-\epsilon}$, the variance is not
and depends on $\arg(\lambda)$. \par
We also determine the cumulant and the probability density function of
the areas $A(\lambda)$ for $\lambda\in \Lambda(R)$. 
\begin{prop}\label{prop:area-distribution}
For \(t\in [0,1] \), define
\[
F_R(t):=
\frac{1}{|\Lambda(R)|}
\#\{\lambda\in \Lambda(R):\ A(\lambda)\le t\}.
\]
Then, as \(R\to\infty\),
\[
F_R(t)\longrightarrow \int_0^t f(a)\,da
\]
where the density function \(f\), for \(0<a<1\) and $z_a=\min(a,1-a)$, is
\begin{align*}
f(a)
=
\frac{2z_a}{\sqrt{1+4z_a^2}}
+
\frac{1}{\sqrt{2z_a}}\int_{2z_a}^{1}
\frac{\sqrt m}{\,(1+m^2)^{3/2}}\,dm.
\end{align*}
\end{prop}
We note that $f(1-a)=f(a)$, and that the above integral can be written
as an
elliptic integral of the second
kind. 
Interestingly, the density function has  two (integrable) singularities
at $a=0,1$, cf. Figure \ref{fig:dens}.
\begin{figure}[H]
    \centering
     \includegraphics[scale=0.7]{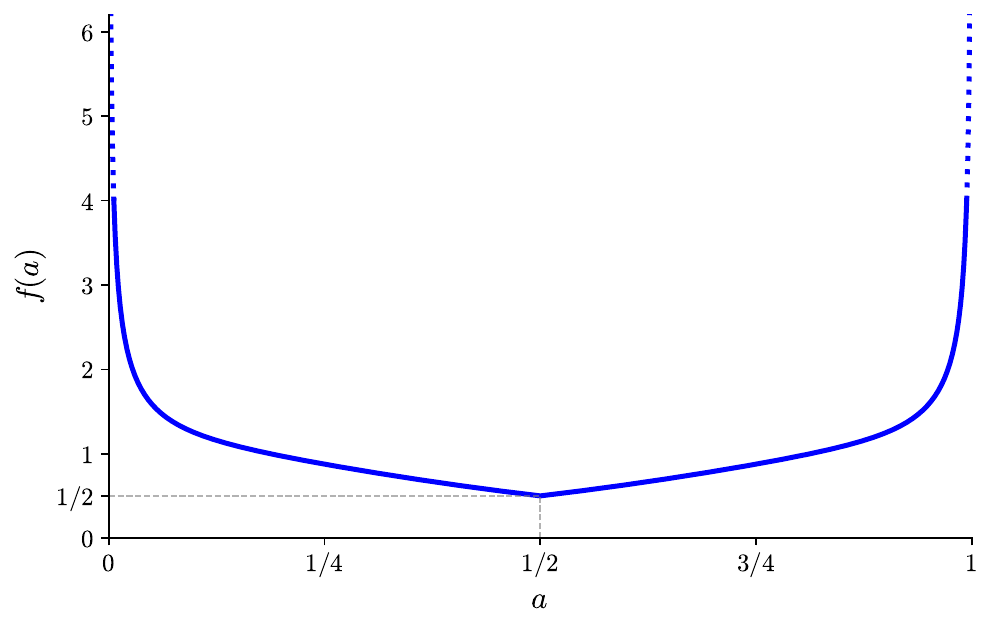}
    \caption{Density function $f(a)$}
    \label{fig:dens}
\end{figure}
The singularity near $a=0$ is asymptotically given by
$$f(a)\sim
\frac{\Gamma^2\left(\frac{3}{4}\right)}{2\sqrt{2\pi}}\frac{1}{\sqrt{a}},$$ 
and similarly for $1-a$ small.
\subsection{Discussion}
\label{sec:discussion}

Our investigations were inspired by work of Lester and Wigman who in \cite{LW},
motivated by a probabilistic folklore heuristic for Hardy's
conjectured $O(R^{1/2+\epsilon})$ error term for the Gauss circle
problem (cf. \cite{hardy-average-order}), initiated the study of area
correlations of the intersections of $B(R)$ with unit square
translates of lattice points inside thin annuli.
More precisely, for
$R>0$ large, they considered a different set of lattice points
$\tilde{\Lambda}(R) := \{ \lambda \in \Z^{2} : ||\lambda|-R| <
1/\sqrt{2} \}$, ordered by argument, with associated
areas\footnote{One notable difference from our setup is that
  $\tilde{A}_{\lambda_{i}}=1$ or $\tilde{A}_{\lambda_{i}}=0$ occurs
  for a positive proportion of elements in $\tilde{\Lambda}(R)$; this
  is due to including translates of squares either lying entirely
  inside or outside of $B(R)$. Interestingly this introduces some
  additional randomization compared to our setting (the
  ``extra''/``empty'' squares has a reshuffling effect on the ordering
  of the squares), though it appears difficult to exploit.}
$\tilde{A}_{\lambda_{i}} := |B(R) \cap ( [-1/2,1/2]^{2} +
\lambda_{i})|$ and showed that the area correlation
$$
\widetilde{C_{k}} =
\lim_{R\to \infty}
\frac{1}{|\tilde{\Lambda}(R)|}
\sum_{i=1}^{|\tilde{\Lambda}(R)|}
( \tilde{A}_{\lambda_{i}} \cdot   \tilde{A}_{\lambda_{i+k}} 
-1/4  )
$$
exists for each $k$. They also showed vanishing on average in that
$\lim_{K \to  \infty } \frac{1}{K}\sum_{k=1}^{K}
\widetilde{C}_{k} = 0$, and also conjectured that $\lim_{k\to \infty}
\widetilde{C}_{k} = 0$.

Our approach is in the spirit of dynamical systems, and can be viewed
as a study of a ``circle flow'' on the torus, which in turn can be
given a physical interpretation as the motion of an electron on a
torus with a magnetic field; $R \to \infty$ corresponds to a vanishing
magnetic field. Now, if the flow (or rather, the Poincare map with
respect to the Poincare section given by the sides of the square) was
mixing, vanishing area correlations would immediately
follow. However, as the circle flow can be well approximated by a
linear flow for  long times, mixing certainly does not hold, the
angle of intersection with vertical and horizontal lines remain
essentially constant.

Rather, we exploit a certain ``order stability'': a line segment, or
circular arc, connecting two squares and passing through $k-1$
intermediate squares can be perturbed while preserving the
ordering. This, together with an equidistribution result pertaining to
very small angular variation of tangent lines allows us to decompose
the area correlation sum into parts on which the areas approximately
decorrelate.  More precisely, the area of the intersection
of the first square is essentially constant, while the second area
intersections (roughly) uniformly ``sweeps through'' all areas between
zero and one as the angle varies (in very small intervals.)
Our setup turns out to somewhat be in the spirit of Gauss' original
argument for the circle problem --- he divides the squares
intersecting the ball into 
internal squares and boundary squares, and estimates the number of
boundary squares in terms of the length of the boundary,
c.f. Appendix~\ref{appendixA} for details.

We finally remark that our global mean area estimate $1/2+o(1)$
implies a slight improvement on Gauss $O(R)$ error term, due to a
curious ``perfect cancellation'' property between the boundary terms
of quarter circles, cf. Appendix~\ref{appendixB} for more details. We
also show that error $O(R^{1/2+\epsilon})$, for any $\epsilon>0$,
would follow from a certain IID hypothesis on the distribution of
$\{ A(\lambda)\}_{\lambda \in \Lambda(R)}$.\fixmehide{expand? maybe
  outline the argument?}
\subsection{Acknowledgments}
We would like to thank Steve Lester and Igor Wigman for many helpful
discussions and comments. \par

Parts of the present work were carried out while M.B. was at KTH being partially supported by
the Swedish Research Council (2020-04036) and others while was at the Department of Mathematics, Università di Milano,
as the recipient of a postdoctoral fellowship funded by the ``MUR - Department of Excellence 2023-2027, CUP code G43C22004580005 - project code DECC23\_012\_DIP''.\par
P.K. was partially supported by
the
Swedish Research Council (2020-04036, 2024-04820)

\section{Preliminaries}
\label{sec:prelimiaries}

\subsection{Lattice points near curves}
\label{sec:where-to}

Given a curve $\gamma \subset \R^2$, we say that a lattice point
$\lambda \in \Lambda$ is near $\gamma$ if
$\gamma \cap S_{\lambda} \neq \emptyset$. We begin with some simple
observations regarding the number of lattice points near curves, in
particular near circular arcs.

\begin{lem}
\label{lem:points-near-curve}  
  Let $\gamma$ be a continuous curve having $v_{1},v_{2} \in \R^{2}$
  as endpoints. The number of $\lambda \in \Lambda$ such that $\gamma
  \cap S_{\lambda} \neq \emptyset$ is $\gg |v_{1}-v_{2}|$.

  Further, if $\gamma(t) = (x(t),y(t))$ and both coordinate
  functions $x(t)$ and $y(t)$ are monotone, the number of 
  $\lambda \in \Lambda$ such that $\gamma
  \cap S_{\lambda} \neq \emptyset$ equals $\lVert (\Delta x, \Delta y)
  \rVert_{1} +   O(1)$, where $(\Delta x,\Delta y) = v_{1}-v_{2}$ and
  $\lVert (\Delta x, \Delta y)
  \rVert_{1} = |\Delta x| + |\Delta y|$ denotes the $L^{1}$ norm on
  $\R^2$.
\end{lem}
\begin{proof}
Since  
  $\max( |\Delta x|,|\Delta y|) \gg |v_{1}-v_{2}|$; say
  $|\Delta x| \ge |\Delta y|$ and thus $|\Delta x| \gg
  |v_{1}-v_{2}|$.  By the intermediate value theorem
  $\gamma$ must intersect at least $|\Delta x|-1$ distinct vertical lines
  with integer $x$-coordinates, and these  intersections yield a
  distinct set of  $\lambda$ such that $\gamma \cap S_{\lambda} \neq
  \emptyset$. The argument for $|\Delta x| \le |\Delta y|$ is similar.
  The argument for the second assertion is similar.
\end{proof}

\begin{cor}
\label{lattice-pts-in-small-arcs}  
  If $k = o(R)$, both $\lambda_{j},\lambda_{j+k}$ must be contained in
  a circular sector having arc measure $O(k/R) = o(1)$; in particular,
 $\arg(\lambda_{j+k})-\arg(\lambda_{j}) = o(1)$.

  As for the
  number of lattice points in quadrants and octants we have
  $|\Lambda(R, (3 \pi /2, 2 \pi))| = 2R+O(1)$, and
  $|\Lambda(R, (3 \pi /2, 7 \pi/4))| = R+O(1)$.

  Further, for angular intervals
  $I = [\theta_{1}, \theta_{2}) \subset (3 \pi /2, 2 \pi)$ shrinking
  with $R$ (i.e. $|I| = o(1)$ as $R \to \infty$) we have
  $$
  |\Lambda(R, I)|
  = R \cdot |I| \cdot ( |\cos \theta_{1}| + |\sin \theta_{1}|)
  \cdot (1+o(1)),
  $$
  provided $|I| \cdot R \to \infty$.
\end{cor}

\begin{proof}
  The two first assertions are immediate consequences of the second
  assertion of Lemma~\ref{lem:points-near-curve}.  For the third
  assertion, we note that both the $x$ and $y$ coordinates are
  monotone since the curve segment is contained in the fourth
  quadrant. Since $\Delta x = R \cdot |I| \cos( \theta) \cdot
  (1+o(1))$, and 
  similarly $\Delta y = R \cdot |I| \sin( \theta) \cdot (1+o(1))$ the
  third assertion follows.
\end{proof}

Given lattice points $\lambda_{1},\lambda_{2} \in \Lambda$ we define
the distance between them as
$$
d(\lambda_{1},\lambda_{2)})
:=
\lVert \lambda_{1}-\lambda_{2} \rVert_{1}.
$$
For a circular arc $C$ (of large radius and centered at
the origin) contained in one of the quadrants, starting at some point
in $S_{\lambda_{1}}$ and ending at some point in $S_{\lambda_{2}}$,
the arc traverses exactly $d(\lambda_{1},\lambda_{2})$ distinct
squares $S_{\lambda}$. In particular, if $\lambda_{j},\lambda_{j+k}$
(ordered by argument as described above) lie in the same quadrant, we
have 
$d(\lambda_{j},\lambda_{j+k})=k$; conversely if
$\lambda_{j},\lambda \in \Lambda(R)$ lie in the same quadrant,
with $\arg(\lambda)>\arg(\lambda_{j})$, and
$d(\lambda,\lambda_{j}) =k$, then $\lambda = \lambda_{j+k}$.

\subsection{Reducing to lattice points in the seventh octant and $k>0$}
\label{sec:lattice-points-sixth}

We begin by reducing to the case where angular interval $I$ lies in
the seventh octant, so that the slope of the tangent lines lie in
$(0,1)$; this will be convenient later when we subdivide the circle
into pieces where the tangent slope lies in intervals given by the
Farey dissection of $(0,1)$.

Given an angular interval $I \subset (0,2\pi)$, recall that
$\Lambda(R,I) := \{ \lambda \in \Lambda(R): \arg(\lambda) \in I \} $
and the corresponding local area correlation is given by
$ \corrRIk = \frac{1}{|\Lambda(R,I)|} \sum_{\lambda_{j} \in
  \Lambda(R,I)} (A(\lambda_{j})-1/2) (A(\lambda_{j+k})-1/2) $.
After suitably
subdividing $I$, we may assume that
$I \subset ( \ell \pi/4, (\ell+1) \pi/4)$ for some integer $\ell$, so
that the corresponding lattice points $\{\lambda_{j} \}$ lie in one of the
eight octants (in case $\Lambda(R)$ has points on the coordinate axes
their contribution is negligible since there are at most four such.)
Further, since $\Z^{2}$ is invariant under
rotation by $\pi/2$, and the reflection in the real axis, we may assume
(possibly after changing the sign of $k$) that $I$ lies in the seventh
octant $(3\pi/2, 7\pi/4)$, and consequently that the slope of the
tangent lines of $C(R)$, in said octant, lies in $(0,1)$.

By Corollary~\ref{lattice-pts-in-small-arcs}, both
$\lambda_{j}, \lambda_{j+k}$ are contained in some arc of size
$\ll k$.  Consequently, if $k<0$, there are $O(k)$ lattice points for
which $\lambda_{j+k}$ lies in the sixth octant and $\lambda_{j}$ lies
in the seventh octant, and these give a negligible contribution to the
local correlation sum as long as $I$ is not too small, in particular
if $k/R = o( |I|)$.

Now, if $k<0$ and $\lambda_{j}, \lambda_{j+k}$ both lie in an interval
contained in the seventh octant, we have 
$$\sum_{\lambda_{j} \in
  \Lambda(R,I)} (A(\lambda_{j})-1/2) (A(\lambda_{j+k})-1/2)
$$
$$
=
\sum_{\lambda_{j} \in
  \Lambda(R,I')} (A(\lambda_{j})-1/2) (A(\lambda_{j-k})-1/2)
$$
where $I'=I + O(k/R)$ is some other interval contained in the seventh
octant.  Provided $k/R = o(|I|)$ this allows us to assume that $k>0$
without changing the asymptotics of the local correlation sum.

\section{A probabilistic model for a family of circles}
\label{sec:continuum-model}

We introduce a probabilistic model for the intersections of $C(R)$
with $S_{\lambda}$, for $\lambda$ in the seventh octant of the
circle.
We will view $C(R)$ as oriented counter clockwise, and we are
interested in the distribution of the points where $C(R)$ first enters
each $S_{\lambda}$. As we  restrict $\lambda$ to the seventh
octant, all such points of first intersection (for $R$ sufficiently
large) will be in the lower, or left part of the boundary of
$S_{\lambda}$, and we let
$$
\partial_{L}(S) := \{ (x,0) : x \in [0,1) \} \cup \{ (0,y) : y \in
[0,1) \} \subset S$$ denote the $L$-shaped lower, or left, part of the
boundary of $S$. Similarly we let $\partial_{L}(S_\lambda) $ denote
the lower/left part of the boundary of $S_{\lambda}$; 
we note that $\cup_{\lambda \in \Lambda}\partial_{L}(S_\lambda) $ is a
disjoint union of
$\{(x,y) \in \R^{2} : x \in \Z \text{ or } y \in \Z\}$, the set of
horizontal and vertical lines with  one integer coordinate.

After subtracting by $\lambda$, the intersection
$C(R) \cap S_{\lambda}$ is translated to lie in $S$, and said
translate of $C(R)$ is a circle passing through $\partial_{L}(S)$ at a
point $v(\lambda)$, having a tangent line passing through $v(\lambda)$
with slope $\tan(\theta(\lambda))$, for some angle
$\theta(\lambda) \in (0,\pi/4)$.
Formally, we define $v : \Lambda(R,(3\pi/2, 7\pi/4)) \to
\partial_{L}(S)$ by
$$
v(\lambda) := (C(R) \cap \partial_{L}(S_{\lambda})) - \lambda
$$
(note that the translation of $\partial_{L}(S_{\lambda})$ by
$-\lambda$ equals $\partial_{L}(S)$, and that
$C(R) \cap \partial_{L}(S_{\lambda})$ consists of a single point for
$R$ sufficiently large since $\lambda$ lies in the $7$th octant.) We
further define $\theta : \Lambda(R, (3\pi/2, 7\pi/4)) \to (0,\pi/4)$
as the angle between the tangent line of $C(R)$, at the point of
intersection $C(R) \cap \partial_{L}(S_{\lambda})$, and the horizontal
coordinate axis.

We will next define a probability measure on the said family of
circles passing through the points of $\partial_{L}(S)$; later we will
show that this measure is a good model for the joint distribution of
$(\theta(\lambda),v(\lambda))$ for
$\lambda \in \Lambda(R, (3\pi/2, 7\pi/4))$.

For $\theta \in (0,\pi/4)$ define a  measure $\mu_{\theta}$ on
$\partial_{L}(S)$ by 
$$
\mu_{\theta}(J) =
\begin{cases}
  |J| \cdot
\sin \theta  & \text{if $J \subset
                                                   \{ (x,0) : x  \in
                                                   [0,1)$\}},
  \\
  |J| \cdot
\cos \theta  &  \text{if $J
                                                   \subset \{ (0,y),
                                                   y  \in [0,1)\}$},
\end{cases}
$$
where $J$ denotes an interval, of length $|J|$, contained in either the
vertical or horizontal part of $\partial_{L}(S)$. 
Given a point $v \in \partial_L S$ and an angle $\theta \in (0,\pi/4)$
there is a unique line $L_{v,\theta}$ passing through $v$ having
slope $\tan(\theta)$. Further, there is a unique circle
$C_{\theta,v} = C_{\theta,v}(R)$, of radius $R$, passing through $v$,
with tangent line having slope $\tan \theta$, and whose center lies
{\em above} $L_{v,\theta}$.

We define a probability measure $\mu$ on the collection of circles
$$
\{ C_{\theta,v} : v \in \partial_{L} S, \, \theta \in (0,\pi/4) \}
$$
via 
$$
\mu(I \times J) =
\frac{1}{\pi/4}\int_{I} \mu_{\theta}(J) \, d \theta
$$
for $I \subset (0,\pi/4)$ an interval, and $J \subset \partial_{L} S$
a horizontal or vertical interval.
It will be
convenient to use the shorthand
\begin{equation}
  \label{eq:dmu}
d\mu=
\begin{cases}
  \sin \theta \, d\theta \, dx, \\
  \cos \theta \, d\theta \, dy,
\end{cases}
\end{equation}
where $dx$ is viewed as a measure on the horizontal part of
$\partial_{L}(S)$ and $dy$ as a measure on the vertical
part of $\partial_{L}(S)$.

\subsection{Correlations via the family of circles}
\label{sec:corr-via-family}

To study the joint distribution of $A(\lambda_{j}), A(\lambda_{j+k})$,
as $j$ varies,
we translate each intersection $S_{\lambda_{j}} \cap C(R)$ by
$-\lambda_{j}$, and obtain a family of circles
$C_{v(\lambda_{j}),\theta(\lambda_{j})}$; each translation maps the
pair $(\lambda_{j}, \lambda_{j+k})$ to $(0,\lambda'_{j,k})$ where
$d(\lambda'_{j,k},0)  = k
= \lVert \lambda'_{j,k} \rVert_{1} + O(1)$, and $\lambda'_{j,k}$ lies in
the first octant (as $k>0$.)

Fix $v \in \partial_{L} S$. Given a lattice point
$\lambda$ in the first octant such that
$d(0,\lambda)=k$ (with $k$ large) the set of $\theta \in (0,\pi/4)$
such that $C_{\theta,v} \cap S_{\lambda} \neq \emptyset$ consists of
an open interval $I = (\theta_{1},\theta_{2}) \subset (0,\pi/4)$, with
$|I|=\theta_{2}-\theta_{1} \asymp 1/k = o(1)$, as $k$ grows.

Further, letting $B_{\theta,v}$ denote the ball having boundary
$C_{\theta,v}$, it is straightforward to see that
$$
| B_{\theta,v} \cap S| =
| B_{v,\theta_{1}} \cap S| + o(1) 
$$
if $|\theta-\theta_{1}|=o(1)$. This will  be used to show that the
areas inside $B_{\theta,v}$ and the two squares $S$ and $S_{\lambda}$
{\em approximately decouple} as long as the angle variation is $o(1)$,
in particular that
$$
\int_{I} | B_{\theta,v} \cap S| \cdot |B_{\theta,v} \cap S_{\lambda}|
\, d\theta
=
| B_{v,\theta_{1}} \cap S|  \cdot
\int_{I}  |B_{\theta,v} \cap S_{\lambda}|
\, d\theta
+ o(1)
$$
for $|I|=o(1)$.

To start, we show that knowledge of $\theta$ and $v$, with sufficient
precision and provided we avoid ``corner cases'', uniquely determines
$\lambda$ such that that 
$C_{\theta,v} \cap S_{\lambda} \neq
\emptyset$ with
$d(\lambda,0 ) = k = \lVert \lambda - v \rVert_{1} +O(1)$, as well
as determining the area $|B_{\theta,v} \cap S_{\lambda}|$ up to a
small error.  More precisely, that both $C_{\theta,v}$ and
$C_{v',\theta'}$ intersect $S_{\lambda}$ provided that $|v-v'|=o(1)$,
$|\theta-\theta'| = o(1/|\lambda|)$, and that the intersection
$C_{\theta,v} \cap S_{\lambda}$ is close neither to the upper left
corner, nor the lower right corner, of $S_{\lambda}$. Furthermore, in
this case we also have 
$|B_{\theta,v} \cap S_{\lambda}| = |B_{v',\theta'} \cap S_{\lambda}|
+o(1)$.

We begin by showing that the set of angles $\theta$ for which
$C_{\theta,v}$ intersects $S_{\lambda}$, for $\lambda \in \Lambda$
fixed with $d(0,\lambda)=k$ and $k$ large, depends mildly on the point
$v$.
\begin{lem}
  \label{lem:control-angles}
  Let $\lambda \in \Lambda$ with $|\lambda| =o(R)$ and
  $1/|\lambda|=o(1)$, and fix $v \in \partial_L(S)$. Then
  $I_{v}(\lambda) := \{ \theta : C_{\theta,v} \cap S_{\lambda} \neq
  \emptyset \}$ is an interval $(\alpha,\beta)$ of length
  $|I_{v}(\lambda)| \asymp 1/|\lambda|$. Further, if
  $v' \in \partial_{L}(S)$ satisfies
  $|\lambda|/R \le |v-v'| \le \delta$, then
  $I_{v'}(\lambda) = (\alpha+O(\delta/|\lambda|),
  \beta+O(\delta/|\lambda|))$.
\end{lem}
\begin{proof}
Here and in what follows it will be convenient to identify $\R^{2}$
with $\C$. 
  Let $\theta_{1}<\theta_{2}$ denote angles such that
  $C_{v,\theta_{1}}$ passes through $\lambda+1$ and $C_{v,\theta_{2}}$
  passes through $\lambda+i$, and let
  $L_{1},L_{2}$ denote secant lines passing through $v,
  \lambda_{1}+1$ respectively $v, \lambda_{1}+i$, say having slopes
  $\tan \theta'_1$ respectively $\tan \theta'_2$.  Then
 $\theta'_{1}-\theta_{1} = \arcsin \frac{|\lambda+1-v|}{2R}=\arcsin
 \frac{|\lambda|}{2R}+O(1/R)$ and similarly
 $\theta'_{2}-\theta_{2} = \frac{|\lambda|}{2R}+O(1/R)$.  Thus, since
 $\theta'_{2} - \theta'_{1} = \arg( \lambda+i - v) - \arg(
 \lambda+1-v) \asymp 1/|\lambda|$ we find that $\beta-\alpha =
 \theta_{2}-\theta_{1} = \theta'_{2}-\theta'_{1} +O(1/R) 
 \asymp 1/|\lambda| + O(1/R) \asymp 1/|\lambda|$.

 The argument for the second assertion is similar: let $\alpha'$
 denote the angle such that $C_{v',\alpha'}$ passes through
 $\lambda+1$ and $\beta'$ the angle such that $C_{v',\alpha'}$ passes
 through $\lambda+i$. On letting $L_{1},L_{2}$ denote two secant lines
 passing through $v, \lambda+1$ and $v', \lambda +1$ we find, as
 before, that
 $\alpha'-\alpha = \arg( \lambda+1-v) - \arg( \lambda+1-v') + O(1/R) =
 O\left(|v-v'|/|\lambda| + O(1/R)\right) = O(\delta/|\lambda|)$.  We similarly find
 that $\beta' = \beta + O(\delta/|\lambda|)$.

\end{proof}

\begin{lem}
  Let $I \times J$ be a product of intervals such that
  $C_{\theta,v} \cap S_{\lambda} \neq \emptyset$ for all
  $(\theta,v) \in I \times J$, with $|I|=o(1/|\lambda|)$,
  $|\lambda|=o(R)$, and $|J|=o(1)$, as $R$ grows.  Then, if
  $(\theta',v') \in I \times J$ we have
$$
|B_{\theta,v}
\cap S_{\lambda}|
=
|B_{v',\theta'}
\cap S_{\lambda}|
+o(1)
$$
for all $(\theta,v) \in I \times J$, as $R$ grows.  Further, if
$(\theta(\lambda_{i}),v(\lambda_{i})) \in I \times J$ then
$\lambda_{i+k} = \lambda_{i} + \lambda$.

\end{lem}
\begin{proof}

  Write $\lambda = (x,y)$ and let $\lambda+1 =(x+1,y(v,\theta))$
  denote the first point of intersection between $C_{\theta,v}$
  (starting at $v$ and traversing the circle counter clockwise) and
  the vertical line $L$ passing through $(x+1,0)$. As the family of
  circles $C_{\theta,v}$ intersects the line transversally, with angle
  of intersection
  uniformly bounded away from zero (here we use that
  $|\lambda|  =o(R)$), the function $y(\theta,v)$ is smooth,
  with $ \partial y /\partial v = O(1)$ and
  $\partial y / \partial \theta = O(|\lambda|)$.  Under our
  assumptions  it follows that $y(\theta,v) =
  y(\theta',v')+o(1)$ for 
  all $(\theta,v) \in I \times J$. Noting that the area
  $|B_{\theta,v} \cap S_{\lambda}|$ is a piecewise smooth function of
  $y$, the result follows.

\end{proof}

We next control the contribution of ``corner case'', i.e., the
contribution from $(\theta,v) \in C_{\theta,v} \cap S_{\lambda} \neq
\emptyset$ for {\em some} 
$(\theta,v) \in I \times J$ but not for {\em all }
$(\theta,v) \in I \times J$.
\begin{lem}
  \label{lem:corner-case-negligible}
  Fix $\lambda \in \Lambda$, with $d(0,\lambda) \asymp |\lambda|$
  large and $\arg(\lambda) \in (0,\pi/4)$.
Let  $J \subset \partial_{L}(S)$ be a horizontal
or vertical interval of length one, 
  and let $I_{\lambda}$ be a minimal (open) interval such that the set
  of $(\theta,v)$ 
  for which $C_{\theta,v} \cap S_{\lambda} \neq \emptyset$ is
  contained in $I_{\lambda} \times J$. Then $|I_{\lambda}| \asymp 1/|\lambda|$,  and
we may  choose integers $M = M(R)$, slowly growing with $R$ such
that if we divide $I_{\lambda}$ into $M$ parts and $J$ into $M$ parts,
we 
cover the set
$\{ (\theta,v) : C_{\theta,v} \cap S_{\lambda} \neq \emptyset \}$
with $\asymp M^{2}$ ``internal boxes'' $B_{ij}$,
$1 \le i,j \le M$, having side lengths $\asymp 1/(|\lambda|M)$ and $1/M$,
and the additional  property that 
$ C_{\theta,v} \cap S_{\lambda} \neq \emptyset $ for all
$(\theta,v) \in B_{ij}$, and $O(M)$ ``boundary'' boxes $B_{ij}$ with
the property that $C_{\theta,v} \cap S_{\lambda} = \emptyset $ for 
some $(\theta,v) \in B_{ij}$, and 
$C_{\theta',v'} \cap S_{\lambda} \neq \emptyset$ for some other 
$(\theta',v') \in B_{ij}$.
\end{lem}
\begin{proof}
Given $v \in \partial_{L}(S)$ the set
of $\theta$ such that $C_{\theta,v} \cap S_{\lambda} \neq \emptyset$
is an open 
interval $(\theta_1(v), \theta_{2}(v))$. Let $L_{1}(v), L_{2}(v)$
denote the two lines passing through $v,\lambda+1$ and $v,\lambda+i$;
we note that $C_{v,\theta_{1}(v)} \cap L_{1}(v) = \{v, \lambda+1\}$
and $C_{v,\theta_{2}(v)} \cap L_{2}(v) = \{v, \lambda+i\}$.  Let
$\phi_{1}(v) := \arg( \lambda+1-v)$ denote the angle between $L_{1}$
and the horizontal axis, and let $\phi_{2}(v) := \arg( \lambda+i-v)$
denote the angle between $L_{2}$ and the horizontal axis.  We then
find, as $\arg(\lambda) \in (0, \pi/2)$,  that
\begin{multline}
  \label{eq:phi1-phi2}
\phi_{2}(v) - \phi_{1}(v) = \arg( (\lambda+i-v ) / (\lambda+1-v))
= \\
\arg( (1+i/(\lambda-v))/(1+1/(\lambda-v)))
=  \\ \arg( 1 + (i-1)/(\lambda-v)) + O(1/|\lambda|^{2})
\asymp 1/|\lambda|.
\end{multline}
Similarly, if $v' \in \partial_{L}(S)$, we find that
\begin{equation}
  \label{eq:phiv-phivprime}
\phi_{j}(v)-\phi_{j}(v') \ll |v-v'|/|\lambda| + O(1/|\lambda|^{2})
\end{equation}
for $j=1,2$.

Further, since $\phi_1(v)-\theta_{1}(v)$ is the angle between the
tangent line of $C_{v,\theta_{1}}$ at $v$ and a chord on
$C_{v,\theta_{1}}$, of length $|\lambda+1-v|$, we find that\fixmehide{why?
fix v, need $\gg 1/|\lambda|$}
$\phi_1(v)-\theta_{1}(v) = \arcsin( |\lambda+1-v|/(2R)) =
\arcsin(|\lambda|/(2R)) + O(1/R)$, uniformly for
$v \in \partial_{L}(S)$. Similarly,
$\phi_2(v)-\theta_{2}(v) = \arcsin(|\lambda|/(2R)) + O(1/R)$ also holds
uniformly in $v$.  Consequently, we have
$$
\theta_{j}(v)-\theta_{j}(v') \ll |v-v'|/|\lambda|
+ O
\left(1/|\lambda|^{2} + 1/R \right).
  $$

  In particular, we find that
  $\theta_{j}(v)-\theta_{j}(v') \ll 1/|\lambda|$ which, together with
  $\phi_{2}(v) - \phi_{1}(v) \asymp 1/|\lambda|$, implies that the
  minimal interval $I$ is of size $ \asymp 1/|\lambda|$. Similarly, subdividing
  $I_{\lambda} \times J$ into $\asymp M^{2}$ boxes $B_{ij}$ having side lengths
  $\asymp 1/(M |\lambda|)$ and $1/M$ we find that, for each $i$
  there, are $O(1)$ values of $j$ for which $B_{ij}$ is a boundary
  box, for a total of $O(M)$ boundary boxes, and $\asymp M^{2}$
  internal boxes.
\end{proof}

\subsection{Second intersection area averages}
\label{sec:second-intersection-area-averages}
To show that the average of $|S_{\lambda} \cap B_{\theta,v}|$, as
$\theta$ ranges over elements in the interval
$I = (\theta_1, \theta_{2}) = \{ \theta : C_{\theta,v}\cap S_{\lambda}
\neq \emptyset \} $, equals $1/2+o(1)$ (for $R,|\lambda|$ large) we
approximate $C_{\theta,v}$ by a family of lines passing through $v$,
and in turn approximate these lines by a family of parallel lines.

More precisely, chose $n$ to be a unit vector orthogonal to the vector
$\lambda=\lambda-0$, and let $L$ be the line passing through $\lambda$ in the
direction given by $n$. As $\theta$ ranges over angles in $I$, the
intersection point $w = w(\theta) = L \cap C_{\theta,v}$ can be
written as $w = \lambda + t \cdot n$ for $t = t(\theta)$ ranging over
some interval $T$.
Moreover, since $k$ is growing, we can write
(cf. (\ref{eq:dmu})) $dt = (1+o(1)) \cdot \sin \theta_{1} \, d\theta$
for $x$ fixed (or $dt = (1+o(1)) \cdot \cos \theta_{1} \, d\theta$ for
$y$ fixed). In other words, the pushforward of the measure $\mu$, for
$v=(x,0)$ or $v=(0,y)$ fixed, to the interval $T$ is, up to negligible
error, the uniform measure (up to a scalar depending on $\theta$)

Let $L_{t}$ denote the line through the point $\lambda + t \cdot n$, in
the direction of $\lambda$, and let $H_{L_t}$ denote the half-plane
bounded by $L_{t}$ and containing the center of $C_{\theta,v}$.  We
then have $\arg(w-v) = \theta + O(k/R) = \theta +o(1)$, since the
curvature of $C_{\theta,v} $ equals $1/R$ and
$|v-\lambda| \asymp k = o(R)$, and thus
$|B_{\theta,v} \cap S_{\lambda}| = |H_{L_{t(\theta)}} \cap
S_{\lambda}| + o(1)$.

Moreover, if we let
$T' = \{ t : L_{t} \cap S_{\lambda} \neq \emptyset \} =
(t_{1},t_{2})$, we have $t_{i} = t_{\theta_{i}} + o(1)$ for $i=1,2$.
In particular, the average of $|H_{L_{t(\theta)}} \cap S_{\lambda}|$
for $\theta \in (\theta_{1},\theta_{2})$ is, up to an $o(1)$-error,
the same as the average $|H_{L_{t}} \cap S_{\lambda}|$ for $t \in T'$.
We next show that the average of the area contained in $H_{t}$ and
$S_{\lambda}$, as $t \in T'$, equals $1/2$.
\begin{lem}
\label{lem:1/2}
  Given a unit vector $ u \in \R^{2}$ (not parallel to the axes),
  define a family of parallel lines $\{L_{t}\}_{t \in \R}$ defined by
  $L_{t} \perp u$ and $t \cdot u \in L_{t}$, and a collection of half
  planes $H_{t}$ having $L_{t}$ as the boundary, and $u$ pointing into
  $H_{t}$.  Given a square $S_{\lambda}$ the set
  $T := \{ t : S_{\lambda} \cap L_{t} \neq \emptyset \}$ is a finite
  open interval, and we have
$$
\int_{T} |H_{t} \cap S_{\lambda}| \, dt
=  |T|/2.
$$
\end{lem}
\begin{proof}
  After applying a translation composed with a rotation we may assume
  that the collection of lines $\{L_{t} \}$ are parallel with the
  $x$-axis, and where $S_{\lambda}$ is replaced by a square $S'$, having
  center of mass at the origin (but sides not necessarily parallel
  with the coordinate axes), and $T = [-\tau,\tau]$ for some
  $\tau \in [1/2,1/\sqrt{2}]$. Letting $A(t)$ denote the area above
  $L_{t}$ we find, by reflecting in the $x$-axis, that
  $A_{1}(-t) = 1-A_{1} (t)$.  Thus
  $$
  \int_{-\tau}^{\tau} A_{1}(t) \, dt =
  \int_{0}^{\tau} A_{1}(t) \, dt +
  \int_{0}^{\tau} (1-A_{1}(t)) \, dt
  = \tau = \frac{|T|}{2}.
  $$

\end{proof}

\section{From lattice points to the continuum model}
\label{sec:from-lattice-points}

\subsection{Circle arc approximation via rational slope lines}
\label{sec:circle-line-rational}

We begin by studying the distribution of the intersections $(C(R) \cap
\partial_L(S_{\lambda}))-\lambda$ for $\lambda \in C(R,I)$ and $I$ a
small angular interval.  To do this, we approximate the circular arc,
first by the tangent line, which is then approximated by a line with
rational slope a certain Farey fraction.
Given $x,y \in \R$, let $D(x,y)$ denote the
distance between $x \mod 1$ 
and $y \mod 1$, i.e.,  $D(x,y) := \min_{k \in \Z} |x-y-k|$.

\begin{lem}
\label{lem:approximate-arc-by-line}  
  Let $f \in C^{2}([-T,T])$ be in increasing function such that
  $f(0)=\alpha$ and
  $ \max_{t : |t| \le T}( |f(T) - (\alpha+\beta t)|) \le \epsilon$ for
  some $\beta > 0$.  For integer $k \in [-T,T]$, let $y_{k} = f(k)$,
  and put $\tilde{y}_{k} = \beta k$.  Similarly, for integer
  $l \in [f(-T),f(T)]$, let $x_{l}$ denote the unique solution to
  $f(x) = l$, and similarly let $\tilde{x}_{l}$ denote the unique
  solution to $\alpha+ \beta x = l$.  We then have
  $D(y_{k},\tilde{y}_{k}) < \epsilon$ for all $k \in [-T,T] \cap \Z$.
  Similarly, we have $D(x_{l},\tilde{x}_{l}) \le \epsilon/\beta$ for
  $l \in [f(-T),f(T)] \cap \Z$.
\end{lem}
\begin{proof}
  Since $D(x,y) \le |x-y|$ the first assertion follows from the fact that
  $\max_{t \in [-T,T]} |f(t)-(\alpha+\beta t)| \le \epsilon$.
  As for the second assertion, we note that  $f(x_{l})=l$ and
  $|f(x_{l})-(\alpha+\beta x_{l})| \le \epsilon$ implies that
  $\alpha+\beta x = l$ has 
  a solution within distance at most $\epsilon/\beta$ from $x_{l}$.
  
\end{proof}

\begin{definition}
Let $N,M>0$ be integers.  We say that a sequence
$\{ x_{n} \mod 1\}_{n \leq N}$ is $\epsilon$-equidistributed at
scale
$M$ if for all intervals $J \subset [0,1]$ satisfying $|J| > 1/M$, we
have
$|\{ n \leq N : x_{n} \in J \}| = N \cdot |J| \cdot (1+\theta)$
for some $\theta$ with $|\theta| < \epsilon$.
\end{definition}
For simplicity, we say
that such a  sequence is ``approximately equidistributed''.

After locally approximating $C(R)$ by lines whose slope $\beta$ is
given by Farey fractions, we will be concerned with the intersection
of lines of the form $y= \alpha + \beta x$ for $\beta = p/q$ a
rational number (we tacitly assume $(p,q)=1$) whose height $\max(p,q)$
grows, with horizontal and vertical lines with at least one integer
coordinate.
\begin{lem}
\label{lem:linear-approx-equidistribution}  
Let $0 < p < q$ be coprime integers, and fix $\alpha \in [0,1]$.
If $M \leq q/\log q$ and $N \ge q \log q$ then $\{\alpha+\frac{p}{q} \cdot n
\mod 1\}_{n \leq N}$ is $\epsilon$-equidistributed at scale $M$ for
$\epsilon = O(1/\log q)$.

Similarly, for the $x$-coordinates, $\mod 1$, of the intersections with
horizontal lines we obtain points $x_{m}$ such that
$\alpha + \frac{p}{q} x_{m} = m \in \Z$, i.e.,
$x_{m} = \frac{q}{p} m + \alpha'$ for some $\alpha' \in \R$ and thus
approximate equidistribution holds with $\epsilon = O(1/\log p)$ if
$M \leq p/\log p$ and $N \ge p \log p$.
\end{lem}
\begin{proof}
The sequence
\[
x_n := \left\{\alpha+\frac{p}{q}n\right\}
\]
is periodic modulo $1$ with period $q$. Since $(p,q)=1$, the set of values attained over one period is exactly a translate of the rational grid $q^{-1}\mathbb{Z}/\mathbb{Z}$. Consequently, for any interval $J\subset[0,1]$,
\[
\#\{1\le n\le q : x_n\in J\} = q|J| + O(1),
\]
uniformly in $J$ and $\alpha$.

Write $N = aq + r$ with $0 \le r < q$. Then
\[
\#\{n\le N : x_n\in J\}
=
a\,\#\{1\le n\le q : x_n\in J\} + O(q|J|),
\]
since the final incomplete block contributes at most $O(q|J|)$ points. Using the previous estimate and $a = N/q + O(1)$, we obtain
\[
\#\{n\le N : x_n\in J\}
=
N|J| + O\!\left(\frac{N}{q} + q|J|\right).
\]

Now assume $|J| > 1/M$. From the hypothesis $M \le q/\log q$ we obtain
$q|J| > \log q.$
Also $N \ge q\log q$, hence
\[
\frac{1}{q|J|} \ll \frac{1}{\log q},
\qquad
\frac{q}{N} \ll \frac{1}{\log q}.
\]
Dividing the previous estimate by $N|J|$ therefore gives
\[
\#\{n\le N : x_n\in J\}
=
N|J|\Bigl(1 + O\!\bigl(\tfrac{1}{\log q}\bigr)\Bigr),
\]
uniformly for all intervals $J\subset[0,1]$ with $|J|>1/M$ which is the claimed result.
\end{proof}

\subsection{Rational slope lines and Farey fractions}
\label{sec:some-results-related}
For $Q>0$ an integer let
$F_Q :=\{ p/q : 0 \le p \le q \le Q, (p,q)=1\}$ denote the set of
Farey fractions of height at most $Q$.

Given $p/q \in F_{Q} \setminus \{0,1\}$, define a half open interval
\begin{align}
  \label{eq:int}  
  \begin{split}
    I_{p,q,Q} :&=
[(p+p')/(q+q'), (p+p'')/(q+q''))
\\ &= [p/q - 1/(q(q+q')),p/q + 1/(q(q+q'')) )  
  \end{split}
\end{align}
where $p'/q' < p/q < p''/q''$ are consecutive
elements in $F_{Q}$. The collection of intervals
$\{ I_{p,q,Q} \}_{p/q \in F_{Q} \setminus \{0,1\}}$ is the {\em Farey
  dissection}, and gives a disjoint union of $[1/(Q+1), Q/(Q+1))$
(cf. \cite[\S{3.8}]{hardy-wright-fifth}).  Since the mediant
$(p+p')/(q+q')$ is not in $F_{Q}$, we must have $q+q' > Q$ and thus
$|I_{p,q,Q}| \asymp 1/(qQ)$; in particular  $|I_{p,q,Q}| \gg 1/Q^{2}$.

We begin by bounding the number of lattice points associated with
parts of the circle where the tangent line slope is well approximated by
rationals having low height.
For comparison, we remark that the number of
$\lambda \in \Lambda(R,(3\pi/2, 7 \pi/4))$ such that
$\tan( \arg(\lambda)-3\pi/2 ) \in I$ is $\asymp R |I|$ provided that
$R|I| \gg 1$.

\begin{lem}
\label{lem:bad-slope-negligble}  
  Let $I \subset [0,1]$ be a closed interval. The number of
  $\lambda \in \Lambda(R,(3\pi/2, 7 \pi/4) )$ such that
  $\tan( \arg(\lambda)-3\pi/2 )\in I_{p,q,Q}$, either for some
  integers $0 < p < q \le Q/\log Q$, or $0 < p < q/\log q$ and
  $q \in [Q/\log Q, Q]$, and additionally that $p/q \in I$, is
  $o(R |I|)$ provided that $Q \le \sqrt{R}/\log R$ and
  $|I| \ge (\log \log Q) / \log Q$.
\end{lem}
\begin{proof}
  First note that the number of
$\lambda \in \Lambda(R,(3\pi/2, 7 \pi/4) )$
such that
  $\tan( \arg(\lambda) )\in I_{p,q,Q}$ is $\ll R/(qQ)$ provided that
  $R/(qQ) \gg 1$, which holds given that $q \le Q \le \sqrt{R}/\log R$.
 Further, the number of integer $p \le q$ such that $p/q \in I$ is $q
 |I| + O(1)$. 

 Thus,  the number of  $\lambda$ of the first type (i.e., for which
$p < q \le Q/\log Q$) is
\begin{multline*}
 \ll
 \sum_{q \le Q/\log Q}
 \sum_{p \le  q, p/q \in I}
 (\frac{R}{qQ}+O(1))
 \ll
 \sum_{q \le Q/\log Q}
 (\frac{R}{qQ}+O(1))
 (q |I| + O(1))
\\
\ll
\frac{R |I|}{\log Q}
+ \frac{Q^{2}}{\log^{2} Q} |I|
+ R\frac{\log Q}{Q} 
+ \frac{Q}{\log Q}
=
o(|I|R)
\end{multline*}

since $\frac{\log Q}{Q} = o(|I|)$ and $\frac{Q}{\log Q} = R \frac{Q}{R
\log Q} \ll R \frac{ Q}{Q^{2}\log^{2}R \log Q}= R o(|I|)$.

To bound the number of $\lambda$ of the second type (i.e., for which
$p < q/\log q$ and $ q \in [ Q/\log Q,Q]$) we first note that
the number of $p <  q/\log q$ for which $p/q \in I$ for $q$ fixed is 
$$
\ll \min(  q/\log q, q |I| + O(1))
= q/\log q
$$
for $Q$ sufficiently large since $|I| \ge \log \log(Q)/\log(Q)$
(note that $\log q \asymp \log Q$ for
$q \in [Q/\log Q,Q]$ and $Q$ large.) Thus the number of the second
type of $\lambda$ is
$$
\ll
\sum_{q \in [Q/\log Q,Q]} \frac{q}{\log q} \left( \frac{R}{qQ}+O(1) \right)
\ll
\frac{R}{\log Q} + \frac{Q^{2}}{\log Q}
= o(R |I|) 
$$
since $Q \le \sqrt{R}/\log R$, and $|I| \ge (\log \log Q)/\log
Q$.
 
\end{proof}

\subsection{$\mu$-equidistribution in very short angular intervals}
\label{sec:equid-very-short}

Given $Q$, define
$F_{Q,\text{Good}} := \{ p/q \in F_{Q} : q \in [Q/\log Q, Q], p \in
[q/\log q, q] \}$. We may then prove $\mu$-equidistribution of
$(\theta(\lambda), v(\lambda)$ for $\lambda$ ranging over sets such that 
$(\tan( \theta(\lambda)), v(\lambda))  \in I \times J \subset I_{p,q,Q}
\times J$, for $I,J$ of very small
sizes (i.e., $|I| \asymp (\log R)^{2}/R$) and 
$J \subset \partial_{L}(S)$, with $|J|$ allowed to
slowly shrink.)
\begin{prop}
  \label{prop:small-scale-mu-equidistribution}
  Let $J \subset \partial_{L} S$ be a vertical or horizontal interval,
  and let $I \subset I_{p,q,Q}$ for $p/q \in F_{Q,\text{Good}}$.
  Assume further that $|J| > \log^{4} Q / Q$, for $Q=Q(R) \le \log R$
  tending to infinity with $R$,
  and that
  $|I| \asymp \log^{2} R / R$.  Then the number of
  $\lambda \in C(R,   (3\pi/2, 7\pi/4))$ such that
  $(\theta(\lambda), v(\lambda))  \in I \times J$ is
$$
R \cdot \mu( I \times J) \cdot (1+o(1))
$$
as $R \to \infty$.
\end{prop}
\begin{proof}
  We begin by noting that Lemma~\ref{lem:approximate-arc-by-line} (we
  subdivide $I$ into intervals of length $(Q \log Q) / R$ and can then take
  $T \asymp Q \log Q$, $\epsilon \asymp (Q \log Q)^{2}/R$, and
  $\beta \gg p/q \gg 1/\log Q$) allows us to approximate
  the circular arc by a tangent line with slope
  $\beta \in I_{p,q,Q}$, inducing an error of size $O(Q^{3}/R)$ in
  terms of the intersection with lines having at least one integer coordinate.

  This tangent line, again using
  Lemma~\ref{lem:approximate-arc-by-line} (with $T = Q \log Q$ as
  before, noting that $p/q \gg 1/\log Q$ and
  $|\beta-p/q| \ll \log Q/Q^{2}$ (since $p/q \in F_{Q,good}$) we can in
  turn approximate by a line with slope $p/q \in
  F_{Q,\text{Good}}$, inducing an error of size $O(\log^{3} Q / Q)$ in
  terms of intersections with lines having at least one integer coordinate.
  By
  Lemma~\ref{lem:linear-approx-equidistribution} (with $N = Q \log Q$)
  we obtain that the boundary intersections points
  $\{ v(\lambda) \}_{\lambda \in \Lambda(R,I)}$ are
  $\mu$-equidistributed on $\partial_{L}(S)$ at scale slightly worse than
  $Q/\log^{3}Q$; say $Q/\log^{4}Q $.  
In particular, note that
  $|I|=o(1)$ implies that the density function of the projection of
  $\mu$ onto the $v$-component is, up to negligible error, constant as
  $\theta$ ranges over the interval $I$.
\end{proof}

\section{Proof of Theorem~\ref{thm:vanishing-local-correlations}}
\label{sec:proof-main-thm}

The following main technical result allows us show vanishing area
correlations for $\lambda_{i}$ ranging over lattice points such that
$\lambda_{i+k}-\lambda_{k} = \lambda$ is fixed (this corresponds to
very short angular intervals in case $k$ is large) {\em provided} that
$\tan(\arg(\lambda))$ is not well approximated by rationals of small
height. 
\begin{prop}
  \label{prop:main-technical}
  Given $\lambda \in \Lambda$ in the first octant with
  $d(\lambda,0)=k$ there exist an interval $I(\lambda)$ of length
  $|I(\lambda)| \asymp 1/k$ such that
  $\lambda_{i+k} - \lambda_{i} = \lambda$ implies that
  $\theta(\lambda_{i}) \in I(\lambda)$. If
  $\{ \tan(\theta) : \theta \in I(\lambda)\}$ is contained in some
  $I_{p,q,Q}$ for $p/q \in F_{Q,Good}$, and $Q = Q(R) \leq \min(\log R,
  \sqrt{k}/\log k$),
then
\begin{equation}
  \label{eq:one-square-cancellation}
\sum_{\lambda_{i} : \lambda_{i+k}-\lambda_{i} = \lambda}
(|B(R) \cap S_{\lambda_{i}}|-1/2)
(|B(R) \cap S_{\lambda_{i+k}}|-1/2)
= o(N_{\lambda})
\end{equation}
where
$N_{\lambda} = |\{ \lambda_{i} : \lambda_{i+k}-\lambda_{i} = \lambda
\}| \asymp R/k$.

\end{prop}
\begin{proof}
The interval 
  $I(\lambda) := \{ \theta : C_{\theta,v} \cap S_{\lambda} \neq
  \emptyset \text{ for some $v \in \partial_{L}(S)$} \} $ has the
  property that 
  $\lambda_{i+k} - \lambda_{i} = \lambda$ implies that
  $\theta(\lambda_{i}) \in I(\lambda)$. 
  and by Lemma~\ref{lem:control-angles} we
  have $|I(\lambda)| 
  \asymp 1/k$. 
  
  Let $M = M(R,k)$ denote integers slowly growing with $R,k$, and let
  $J \subset \partial_{L}(S)$ denote an half open interval, either
  vertical or horizontal, of length $1/M$. Covering $\partial_{L}(S)$
  by $2M$ such intervals (i.e., $M$ horizontal and $M$ vertical ones), we
  claim that it suffices to show that the localized second
  intersection area sums are small in the sense that
\begin{equation}
  \label{eq:v-theta-localized-sums}
\sum_{\lambda_{i} : \lambda_{i+k}-\lambda_{i} = \lambda,
  v(\lambda_{i}) \in J}
(|B(R) \cap S_{\lambda_{i+k}}|-1/2)
= o(N_{\lambda}/M).
\end{equation}

To justify the claim, first note that the variation of the first
intersection areas $|B(R) \cap
S_{\lambda_{i}}|$, as $\lambda_{i}$ ranges over elements such that
$v(\lambda_{i}) \in J$ and $\theta(\lambda_{i}) \in
I(\lambda)$, is $o(1)$. (In fact, this holds as long as
$v(\lambda_{i})$ and 
$\theta(\lambda_{i})$ varies in intervals of length $o(1)$.) Summing
the contribution from the $2M$ such intervals $J$ we arrive at an
error term of size  $o( N_{\lambda})$ in (\ref{eq:one-square-cancellation}).

To show that \eqref{eq:v-theta-localized-sums} holds we argue as
follows: let
$I(\lambda,J) := \{ \theta : C_{\theta,v} \cap S_{\lambda} \neq
\emptyset \ \forall v \in J\}$; Lemma~\ref{lem:corner-case-negligible}
then gives $|I(\lambda,J)| \asymp 1/k$. Further, again by
Lemma~\ref{lem:corner-case-negligible} we may cover $I(\lambda,J)$ by
$M$ disjoint (half-open) same sized intervals $\{I_{j} \}_{j \le M}$,
with $|I_{j}| \asymp 1/(kM)$ and where all but $O(1)$ of these $M$
intervals are internal in the sense that
$C_{\theta,v} \cap S_{\lambda} \neq \emptyset $ for {\bf all}
$(\theta,v) \in I_j \times J$. In particular,
$\lambda_{i+k} - \lambda_{i}=\lambda$ holds if
$(\theta(\lambda_{i}), v(\lambda_{i})) \in I_{j} \times J$ and $I_{j}$
is internal.

Since there are $O(1)$ non-internal intervals $I_{j}$, each of length
$O(1/(kM))$,  $|J|=1/M$, and $\mu$-equidistribution holds at
scale $I_{j} \times J$, the corresponding contribution to
(\ref{eq:v-theta-localized-sums}) is
$O(R/(kM^{2})) = o(N_{\lambda}/M)$.  Further, again since
$\mu$-equidistribution holds at scale $I_{j} \times J$, we find that
for the internal intervals $I_{j}$, we have
\begin{align*}
\sum_{\lambda_{i} : \lambda_{i+k}-\lambda_{i} = \lambda,
  (\theta(\lambda_{i},v(\lambda_{i})) \in I_{j} \times J}
&(|B(R) \cap S_{\lambda_{i+k}}|-1/2)
=
\\&R (1+o(1)) \int_{I_{j}\times J} (|C_{\theta,v} \cap S_{\lambda}|-1/2)
d \mu(\theta,v)
\end{align*}
Summing over all $j$ 
we find that
\begin{align*}
\sum_{\lambda_{i} : \lambda_{i+k}-\lambda_{i} = \lambda, \,
  \theta(\lambda_{i}) \in  J}
&(|B(R) \cap S_{\lambda_{i+k}}|-1/2)
=
\\ &R(1+o(1)) 
\int_{I(\lambda,J) \times J} (|C_{\theta,v} \cap S_{\lambda}|-1/2)
d \mu(\theta,v)
\end{align*}
since the contribution from non internal $I_{j}$ is negligible, in the
sum as well as in the integral.
Moreover, with $X \Delta Y$ denoting the symmetric difference between
two sets $X,Y$, we have
$$
\mu( (I(\lambda,J) \times J) \Delta
\{ (v,\theta) : v \in J, C_{v,\theta} \cap S_{\lambda} \neq
\emptyset\}) =
o(\mu( I(\lambda,J) \times J) ).
$$
and thus
\begin{align}
  \label{eq:I-lambda-J-integral}
  \begin{split}
\int_{I(\lambda,J) \times J} &(|C_{\theta,v} \cap S_{\lambda}|-1/2)
d \mu(\theta,v) =\\ &
(1+o(1))
\int_{  \{ (v,\theta) : v \in J, C_{v,\theta} \cap S_{\lambda} \neq
  \emptyset\}  }
(|C_{\theta,v} \cap S_{\lambda}|-1/2)
d \mu(\theta,v)
 \end{split}
\end{align}
which, in case $J$ is a horizontal interval, equals
\begin{equation}
  \label{eq:horisontal-average}
(1+o(1)) 
\int_{\{ (\theta, x) : C_{\theta,(x,0)} \cap S_{\lambda} \neq \emptyset
  , (x,0) \in J \}}
(|C_{\theta,(x,0)} \cap S_{\lambda}|-1/2)
\sin(\theta)  \, d \theta
\, dx.
\end{equation}
Since the variation of $\sin( \theta)$ is $o(1)$ for
$\theta \in I(\lambda,J)$, we obtain, by the discussion in
Section~\ref{sec:second-intersection-area-averages}, that for $x$
fixed,
\begin{multline*}
\int_{\{\theta : C_{\theta, (x,0)} \cap S_{\lambda} \neq \emptyset\}}
(|C_{\theta,(x,0)} \cap S_{\lambda}|-1/2)
d \theta
\\
=
o 
\left(
  |\{\theta : C_{\theta, (x,0)} \cap S_{\lambda} \neq \emptyset \}|
\right)
 =o(1/k)
\end{multline*}
and thus (\ref{eq:I-lambda-J-integral}) is
$
o(1/kM)$ and consequently (\ref{eq:v-theta-localized-sums}) is
$o(R/(kM)) = o(N_{\lambda}/M)$.
The argument for $J$ a vertical interval is similar.

\end{proof}

\begin{cor}
\label{cor:correlation-vanish-good-large-intervals}  
  Let $I \subset (3\pi/2, 7\pi/4)$ be an interval such that
  $\{ \tan(\theta - 3 \pi /2) : \theta \in I\} = I_{p,q,Q}$ for
  $p/q \in F_{Q,good}$.  Provided $1/k = o(1/(qQ))$ we have
$$
\sum_{\lambda_{i} : \arg(\lambda_{i}) \in I  }
(|B(R) \cap S_{\lambda_{i}}|-1/2)
(|B(R) \cap S_{\lambda_{i+k}}|-1/2)
= o( R | I|)
$$
\end{cor}
\begin{proof}
  Since $|I| \asymp |I_{p,q,Q}| \asymp 1/(qQ)$ we may cover $I$ by
  intervals $I(\lambda)$, for $\lambda$ ranging over some subset of
  $\{ \lambda' \in \Lambda : d(0,\lambda')=k\}$. Further, since
  $|I_{\lambda}| \asymp 1/k$, the number of such intervals grows with
  $R$, with all but $O(1)$ them being contained in $I$. The result now
  follows from Proposition~\ref{prop:main-technical} by summing over
  the relevant set of $\lambda$'s on noting that the contribution from
  intervals $I(\lambda)$ not contained in $I$ is negligible (since there
  are only $O(1)$ such)

\end{proof}

By taking  longer intervals $I$ we can allow intervals  for which
the tangent line is well approximated by rationals of low height.

\begin{cor}
  Let $I \subset (3\pi/2, 7\pi/4)$ be an interval such that
  $|I| > (\log \log Q) / \log Q $, with
  $Q = \min(\log R, \sqrt{k}/\log k$). Then
$$
\sum_{\lambda_{i} : \arg(\lambda_{i}) \in I  }
(|B(R) \cap S_{\lambda_{i}}|-1/2)
(|B(R) \cap S_{\lambda_{i+k}}|-1/2)
= o( R | I|)
$$
\end{cor}
\begin{proof}
  Since we assume that $|I| >  (\log \log Q) / \log Q$,
  Lemma~\ref{lem:bad-slope-negligble} implies that the contribution
  from $\lambda_{i}$ such that $\tan(\arg(\lambda_{i})) \in I_{p,q,Q}$
  for $p/q \not \in F_{Q,good}$ is $o(R|I|)$.

 Since $1/k = o(1/Q^{2})$ the result follows by
  covering the remaining part of $I$ by intervals $I'_{p,q,Q} = \{
  \theta : \tan(\theta) \in I_{p,q,Q}\}$ with $p/q \in F_{Q,good}$,
  and applying
  Corollary~\ref{cor:correlation-vanish-good-large-intervals}.

\end{proof}

Since the discussion in Section~\ref{sec:lattice-points-sixth} allows
us to reduce to intervals contained in the seventh octant,
Theorem~\ref{thm:vanishing-local-correlations} is now an immediate
consequence of the previous corollary.
\section{Local area mean}\label{sec:localmean}
For an interval \(I\subset(0,2\pi)\) 
we define the local area mean as
\[
\mean(R,I):=\frac{1}{|\Lambda(R,I)|}\sum_{\lambda\in\Lambda(R,I)}A(\lambda);
\]
recall that
$\Lambda(R,I):=\{\lambda\in\Lambda(R):\arg(\lambda)\in I\}$.  The
following two results  give an asymptotic for the local mean in
two ranges. The first is uniform in $\arg(\lambda)$, at the cost of a
larger interval, while the second is based on the Diophantine type of
$\tan(\arg(\lambda))$ and allows taking much smaller intervals ``in generic
position''.
\subsection{Uniform local mean}
We first prove the uniform result.
\begin{prop}
\label{prop:meanuni}
If
$$|I| >  \frac{(\log R)^{8}}{\sqrt{R}},$$
then, as $R \to \infty$,
    $$\mean(R,I)=\frac{1}{2}+o(1).$$
  \end{prop}
  We remark that the interval length assumption is essentially sharp:
  near the lowest part of the circle we must traverse more than
  $R^{\frac{1}{2}}$ squares to obtain equidistribution of the hitting
  points on the left side of the fundamental domain.
We start with a variation of Lemma \ref{lem:bad-slope-negligble}
for large $Q$; to do so we take 
$p \in [\log^{4} q, q]$ in our definition of ``good'' Farey
fractions in order to control the contribution from small height
intervals $I_{p,q,Q}$.
\begin{lem}
\label{lem:bad-slope-negligble1}  
  Let $I \subset [0,1]$ be a closed interval. The number of
  $\lambda \in \Lambda(R,(3\pi/2, 7 \pi/4) )$ such that
  $\tan( \arg(\lambda)-3\pi/2 )\in I_{p,q,Q}$, either for some
  integers 
  $$0 < p < q \le Q/\log Q,$$
  or 
  $$0 < p < (\log q)^{4}, \quad
q \in [Q/\log Q, Q],
  $$
 and additionally that $p/q \in I$, is
  $o(R |I|)$ provided that 
  $$Q = \frac{\sqrt{R}}{(\log R)^{3}}$$ 
  and
  $$|I| >   \frac{(\log R)^{8}}{\sqrt{R}}.$$
\end{lem}
\begin{proof}
   The proof is identical to Lemma \ref{lem:bad-slope-negligble}, with the modified range for $p$.
\end{proof}
We next allow for much smaller $p$ in the definition of
$F_{Q,\text{Good}}$ and prove a variation of Proposition
\ref{prop:small-scale-mu-equidistribution}. Given $Q$, define
$$\widehat F_{Q,\text{Good}} := \{ p/q \in F_{Q} : q \in [Q/\log Q, Q], p \in
[(\log q)^{4}, q] \}.$$
The main differences compared to Proposition
\ref{prop:small-scale-mu-equidistribution} is that we need to ensure
that $I$ contains some slope $\beta$ such that $\beta \in
I_{p,q,Q}$  
for some $\frac pq\in F_{Q,\mathrm{Good}}$.
\begin{prop}
  \label{prop:small-scale-mu-equidistribution1}
  Let $J \subset \partial_{L} S$ be a vertical or horizontal interval,
  and assume that 
$$\exists\, \beta \in I_{p,q,Q} \cap I \quad \textit{for some}\quad \frac{p}{q} \in \widehat F_{Q,\text{Good}}.$$
  Assume further 
  $$|J| > \frac{1}{\log\log R},$$
  $$Q=Q(R) =\frac{\sqrt{R}}{(\log R)^{3}},$$
   and 
  $$|I| =  \frac{1}{\sqrt{R}\log R}.$$  
  Then the number of
  $\lambda \in \Lambda(R,   (3\pi/2, 7\pi/4))$ such that
  $(\theta(\lambda), v(\lambda))  \in I \times J$ is
$$
R \cdot \mu( I \times J)+o(R|I||J|).
$$
as $R \to \infty$.
In fact, the proof gives the stronger asymptotic 
$$
R \cdot \mu( I \times J)(1+o(1))
$$
if $J$ is a vertical interval, or if
$I \subset [1/(\log \log R)^{2},\pi/4]$.
\end{prop}
\begin{proof}
  Here we make  a cutoff at $\theta(\lambda)= 1/(\log \log R)^2$ and treat
  the part with $\theta(\lambda)$ small differently.
  First, for $\theta(\lambda) > 1/(\log \log R)^2$ we argue as in the
  proof of Proposition \ref{prop:small-scale-mu-equidistribution}. In
  the smaller range, we distinguish between the lower and left side
  intersections. The number of lower side intersections is $\ll R
  \cdot |I|/(\log \log R)^2=o(R|I||J|)$.
  For the left side intersections, we again argue as in the proof of
 Proposition \ref{prop:small-scale-mu-equidistribution}.
\end{proof}
\begin{proof}[Proof of Proposition \ref{prop:meanuni}]
By the reduction in Section~\ref{sec:lattice-points-sixth} we may assume
\[
I \subset (3\pi/2,7\pi/4), \qquad \widehat I := I - \frac{3\pi}{2} \subset (0,\pi/4).
\]
Let \( Q = \sqrt{R}/(\log R)^{3} \). By Lemma~\ref{lem:bad-slope-negligble1}, the number of
\(\lambda \in \Lambda(R,I)\) such that
$
\tan\!\left(\arg(\lambda)-\frac{3\pi}{2}\right) \in I_{p,q,Q}
$
for some \( p/q \notin \widehat{F}_{Q,\mathrm{Good}} \) is \( o(R|I|) \). Since
\( A(\lambda) \le 1 \), their contribution is negligible.
Partition \(\widehat I \) into disjoint intervals $ \widehat I_\nu$ such that 
$$ |\widehat I_\nu|=\frac{1}{\sqrt R\log R},$$
up to \(O(1)\) endpoint intervals. The total contribution from the endpoints is negligible.
If \(\widehat I_\nu\) contains no good slope, which means
$$\forall\, \frac{p}{q} \in \widehat F_{Q,\text{Good}} \,\,\nexists\, \beta \in I_{p,q,Q} \cap \widehat I_\nu ,$$
then its contribution is negligible by Lemma
\ref{lem:bad-slope-negligble1} . Thus it suffices to consider those
intervals $\widehat I_\nu$ containing some good slope, i.e., that
$$\exists\, \beta \in I_{p,q,Q} \cap \widehat I_\nu\quad  \textrm{for some} \quad  \frac{p}{q} \in \widehat F_{Q,\text{Good}}.$$

Fix such an interval \(\widehat I_\nu\). By Proposition \ref{prop:small-scale-mu-equidistribution1}, for any interval \(K\subset \partial_L(S)\) with
\[
|K|>\frac1{\log \log R},
\]
we have
\[
\#\{\lambda:(\theta(\lambda),v(\lambda))\in \widehat I_\nu\times K\}
=
R\,\mu(\widehat I_\nu\times K)+o(R|\widehat I_\nu||K|).
\]

Summing over a partition of \(\partial_L(S)\) into such intervals \(K\), we find that
\[
\sum_{\substack{\lambda\in\Lambda(R,I)\\ \theta(\lambda)\in \widehat I_\nu}}
A(\lambda)
=
R
\int_{\widehat I_\nu\times\partial_L(S)}
|B_{\theta,v}\cap S|\,d\mu+o(R|\widehat I_\nu|),
\]
and similarly for the normalization factor
\[
\sum_{\substack{\lambda\in\Lambda(R,I)\\ \theta(\lambda)\in \widehat I_\nu}}
1
=
R
\int_{\widehat I_\nu\times\partial_L(S)}
d\mu+o(R|\widehat I_\nu|).
\]

By Lemma ~\ref{lem:1/2},
\[
\int_{\widehat I_\nu\times\partial_L(S)}
|B_{\theta,v}\cap S|\,d\mu
=
\frac12
\int_{\widehat I_\nu\times\partial_L(S)} 1 \,
d\mu.
\]

Hence
\[
\sum_{\substack{\lambda\in\Lambda(R,I)\\ \theta(\lambda)\in \widehat I_\nu}}
A(\lambda)
=
\frac12
\#\{\lambda\in\Lambda(R,I):\theta(\lambda)\in \widehat I_\nu\}+o(R|\widehat I_\nu|).
\]
Summing over all good intervals and adding back the negligible bad and endpoint contributions, we find that
\[
\sum_{\lambda\in\Lambda(R,I)} A(\lambda)
=
\frac12 |\Lambda(R,I)| + o(R|I|).
\]

Since $|\Lambda(R,I)|\asymp R|I|, $
the desired result follows.
\end{proof}

\subsection{Local mean in  good Farey intervals}
We now prove a local mean result based on the Diophantine type of
$\tan(\arg(\lambda))$. As in Section~\ref{sec:lattice-points-sixth},
we focus on the seventh octant.
\begin{prop}
Let
\[
I=I(R)\subset (3\pi/2,7\pi/4)
\]
be an interval such that
\[
\tan\!\left(I-\frac{3\pi}{2}\right)\subset I_{p,q,Q}
\]
for some $p/q\in F_{Q,\mathrm{Good}}$, with $Q=Q(R)\to\infty$, $Q\le \log R$, and
\[
|I|\asymp \frac{\log^2 R}{R}.
\]
Then
\[
\mean(R,I)=\frac12+o(1).
\]
\end{prop}
\begin{proof}
    The result follows, as in the proof of Proposition \ref{prop:meanuni}, using Lemma \ref{lem:bad-slope-negligble}, Proposition \ref{prop:small-scale-mu-equidistribution} and Lemma \ref{lem:1/2}.
\end{proof}
\section{Local and global area variance}

Given an interval \(I\subset(0,2\pi)\) we  define the local variance
\begin{equation}
  \label{eq:local-variance-def}
\mathrm{Var}(R,I):=\frac{1}{|\Lambda(R,I)|}\sum_{\lambda\in\Lambda(R,I)}
\bigl(A(\lambda)-\mean(R,I)\bigr)^2.
\end{equation}
As in Section~\ref{sec:lattice-points-sixth}, after subdividing into
octants and using the symmetries of \(\mathbb Z^2\), we may restrict
attention to intervals contained in the seventh octant. 

\begin{lem}\label{lem:local-average-variance}
Let  \(I\subset(3\pi/2,7\pi/4)\) and assume that 
  $$|I| >  \frac{(\log R)^{8}}{\sqrt{R}}$$
and put
$
\hat{I}:=I-\frac{3\pi}{2}\subset(0,\pi/4).
$
\fixmehide{repetition M. removed the full line}
Then
for every bounded continuous function 
\[
F:(0,\pi/4)\times\partial_L(S)\to\mathbb R
\]
we have
\[
\frac{1}{|\Lambda(R,I)|}\sum_{\lambda\in\Lambda(R,I)}F(\theta(\lambda),v(\lambda))
=
\frac{\displaystyle\int_{\hat{I}\times\partial_L(S)}F(\theta,v)\,d\mu(\theta,v)}
{\displaystyle\int_{\hat{I}\times\partial_L(S)}d\mu(\theta,v)}
+o(1),
\]
where
\[
d\mu(\theta,v)=
\begin{cases}
\sin\theta\,d\theta\,dx, & v=(x,0),\\[1ex]
\cos\theta\,d\theta\,dy, & v=(0,y).
\end{cases}
\]
\end{lem}
\begin{proof}
  We give a sketch based on the proof of Proposition
  \ref{prop:meanuni}. We partition \(\hat{I}\times\partial_L(S)\) into
  admissible shrinking boxes \(Q=I\times J\) (i.e., so we may apply
  Proposition~\ref{prop:small-scale-mu-equidistribution1}), and
  replace \(F\) on \(Q\) by its value at a representative point. The
  counting measure on lattice points pushes forward to \(\mu\),
  uniformly over the partition, and the resulting Riemann sums
  converge to the corresponding \(\mu\)-integral.  Applying the same
  argument to \(F\equiv1\) gives the denominator.
\end{proof}

\begin{lem}\label{lem:fixed-theta-variance}
For \(0<\theta<\pi/4\), define
\[
g(\theta):=\frac{\tan^3\theta-5\tan^2\theta+15\tan\theta+5}{60(1+\tan\theta)}.
\]
Then
\[
\int_{\partial_L(S)}
\Bigl(|B_{\theta,v}\cap S|-\frac12\Bigr)^2\,d\mu_\theta(v)
=
(\sin\theta+\cos\theta)\,g(\theta).
\]
\end{lem}

\begin{proof}
  We need an explicit expression for the area in the fundamental
  domain above a specific line.  Fix \(0<\theta<\pi/4\), and write
  \(m=\tan\theta\). A line of slope \(m\) is written $ y=mx+b.  $ The
  pushforward of \(d\mu_\theta\) to the $y$-intercept parameter $b$ is
\[
d\mu_\theta\mapsto \cos\theta\,db
\qquad\text{on }[-m,1].
\]

Let \(A(b,\theta)\) denote the area of the part of \(S=[0,1]^2\) lying above the line \(y=mx+b\).
Then
\begin{align}
\label{eq:A}
\begin{split}
A(b,\theta)=
\begin{cases}
1-\dfrac{(m+b)^2}{2m}, & -m\le b\le 0,\\[2ex]
1-b-\dfrac m2, & 0\le b\le 1-m,\\[2ex]
\dfrac{(1-b)^2}{2m}, & 1-m\le b\le 1.
\end{cases}
\end{split}
\end{align}
Therefore
\[
\int_{\partial_L(S)}
\Bigl(|B_{\theta,v}\cap S|-\frac12\Bigr)^2\,d\mu_\theta(v)
=
\cos\theta\int_{-m}^{1}\Bigl(A(b,\theta)-\frac12\Bigr)^2\,db.
\]
A direct computation on the three intervals gives
\[
\int_{-m}^{1}\Bigl(A(b,\theta)-\frac12\Bigr)^2\,db
=
\frac{m^3-5m^2+15m+5}{60}.
\]
\end{proof}

\begin{prop}[Local Variance]\label{thm:local-variance}
Let 
  $$|I| >  \frac{(\log R)^{8}}{\sqrt{R}}.$$
Assume that \(I\subset(3\pi/2,7\pi/4)\) and write
$
\hat{I}:=I-\frac{3\pi}{2}\subset(0,\pi/4).
$
Then
\[
 \mathrm{Var}(R,I)
=
\frac{\displaystyle\int_{\hat{I}}g(\theta)(\sin\theta+\cos\theta)\,d\theta}
{\displaystyle\int_{\hat{I}}(\sin\theta+\cos\theta)\,d\theta}
+o(1).
\]
If in addition $|\hat{I}|\to 0$ and $\theta_0\in \hat{I}$, then
\[
\mathrm{Var}(R,I)\to g(\theta_0).
\]
\end{prop}

\begin{proof}
  By (\ref{eq:local-variance-def}) and the local mean result
  (cf. Proposition~\ref{prop:meanuni}), we have
\[
\mathrm{Var}(R,I)
=
\frac{1}{|\Lambda(R,I)|}
\sum_{\lambda\in\Lambda(R,I)}
\Bigl(A(\lambda)-\frac12\Bigr)^2
+o(1).
\]
Lemma~\ref{lem:local-average-variance} then yields
\[
\mathrm{Var}(R,I)
=
\frac{\displaystyle\int_{J_R\times\partial_L(S)}
\Bigl(|B_{\theta,v}\cap S|-\frac12\Bigr)^2\,d\mu(\theta,v)}
{\displaystyle\int_{J_R\times\partial_L(S)}d\mu(\theta,v)}
+o(1).
\]
Now integrate first in \(v\). By Lemma~\ref{lem:fixed-theta-variance},
\[
\int_{\partial_L(S)}
\Bigl(|B_{\theta,v}\cap S|-\frac12\Bigr)^2\,d\mu_\theta(v)
=
(\sin\theta+\cos\theta)\,g(\theta).
\]
Therefore
\[
\int_{J_R\times\partial_L(S)}
\Bigl(|B_{\theta,v}\cap S|-\frac12\Bigr)^2\,d\mu(\theta,v)
=
\int_{J_R}g(\theta)(\sin\theta+\cos\theta)\,d\theta,
\]
while
\[
\int_{J_R\times\partial_L(S)}   1 \, d\mu(\theta,v)
=
\int_{J_R}(\sin\theta+\cos\theta)\,d\theta.
\]
Substituting these identities proves the first claim. The second follows by continuity of \(g\).
\end{proof}

\begin{cor}[Global Variance]\label{cor:global-variance}
We have
\[
\mathrm{Var}(R,(3\pi/2,7\pi/4))\to c_{\mathrm{glob}},
\]
where
\[
c_{\mathrm{glob}}
=
\frac{13}{60}-\frac{1}{12}\log(1+\sqrt2)-\frac{1}{30\sqrt2}.
\]
\end{cor}

\begin{proof}
Applying Proposition~\ref{thm:local-variance} and implementing the substitution \(u=\tan\theta\), gives
\[
c_{\mathrm{glob}}
=
\frac{1}{60}\int_0^1\frac{u^3-5u^2+15u+5}{(1+u^2)^{3/2}}\,du.
\]
A direct computation gives the desired result.
\end{proof}
\section{Limiting distribution}
\begin{proof}[Proof of Proposition \ref{prop:area-distribution} ] Following the same lines of the proof of Proposition \ref{thm:local-variance} the result is reduced to the following computation.  
Fix \(0<\theta<\pi/4\), and write
$
m=\tan\theta.
$
Let \(A=A(b,\theta)\) denote the area of the part of \(S=[0,1]^2\) lying above the line \(y-mx=b\). Then, by \eqref{eq:A},
it follows that the conditional density of \(A\) for fixed \(m\) is
\[
f_m(a)=
\begin{cases}
\dfrac{\sqrt m}{(1+m)\sqrt{2a}}, & 0<a<\dfrac m2,\\[3ex]
\dfrac{1}{1+m}, & \dfrac m2<a<1-\dfrac m2,\\[3ex]
\dfrac{\sqrt m}{(1+m)\sqrt{2(1-a)}}, & 1-\dfrac m2<a<1,
\end{cases}
\]
and in particular \(f_m(a)=f_m(1-a)\).

We now average over \(\theta\). Since
\[
\int_0^{\pi/4}(\sin\theta+\cos\theta)\,d\theta=1,
\]
the angle variable is distributed with density \((\sin\theta+\cos\theta)\,d\theta\). Passing to
$m=\tan\theta,$
we have
\[
d\theta=\frac{dm}{1+m^2},
\qquad
\sin\theta+\cos\theta=\frac{1+m}{\sqrt{1+m^2}},
\]
so the measure becomes
\[
\frac{1+m}{(1+m^2)^{3/2}}\,dm,
\qquad 0<m<1.
\]
Hence
\[
f(a)=\int_0^1 f_m(a)\,\frac{1+m}{(1+m^2)^{3/2}}\,dm.
\]
Substituting the formula for \(f_m(a)\) and simplifying gives, for \(0<a<\tfrac12\),
\[
f(a)
=
\int_0^{2a}\frac{dm}{(1+m^2)^{3/2}}
+
\int_{2a}^{1}\frac{\sqrt m}{\sqrt{2a}\,(1+m^2)^{3/2}}\,dm.
\]
Since each \(f_m\) is symmetric, so is \(f\), and therefore
$f(a)=f(1-a)$ for all $0<a<1$. This proves the formula for the density.
\end{proof}

\section{Further questions}

The results of this paper suggest several possible extensions. First,
one may replace the square fundamental domain $S=[0,1)^2$ by a
different fundamental domain for $\mathbb{Z}^2$, or more generally for
an arbitrary planar lattice. The case of polygonal fundamental domains
appears particularly interesting: the area functional is still
piecewise analytic, but the sequence of sides through which the curve
enters may have a more complicated symbolic structure. In particular,
the order-stability property does not always hold.  

A second direction is to replace the disk by other convex domains. For
strictly convex domains whose boundary has curvature uniformly bounded
way from zero, we would expect similar qualitative behavior.

A basic problem is to characterize domains (and fundamental domains)
which have asymptotically vanishing
area correlations.

\appendix

\section{A historical overview}

\label{appendixA}
\label{sec:gauss} We give a short historical overview
of Gauss original formulation for the circle problem.
In his study of binary quadratic forms \cite[I., art. 4]{Gauss}, he
introduced what is now known as ``the Gauss circle problem'', and
states that if the area of a figure scales as $k^2$ and the length of
the perimeter scales as $k$ then the number of the lattice points
inside  asymptotically equals  the area.  He gives 
the 
circle  as an example of a figure, with the asymptotic 
$|B(R) \cap \Z^2| \sim \pi R^2,$
but provides  no error term. Furthermore the proof is (roughly)
``suppressed, hastening to more difficult matters''
\begin{quote}
  Ceterum si operae pretium esset, facile demonstrationem illius
  theorematis antiquo rigore absolvere possemus, quam tamen hocce
  quidem loco supprimere maluimus ad difficiliora properantes.
\end{quote}

In \cite[II., art. 2]{Gauss} Gauss states that a figure delimited by a
closed curve can be divided into squares inside the figure and squares
intersecting the perimeter. He then states that the area of the figure
can be approximated up to any degree of precision by the area of the
internal squares, taking small enough squares (with some implicit
assumptions on the uniformity of the boundary of the figure). To prove
this he shows that the number of squares of side $a$ intersected by a
(closed) curve or length $l$ is at most $4\left(1+\frac{l}{a}\right)$.
This shows that the area of the intersected squares is negligible as $a
\to 0$. 
In \cite[II.,
art. 3]{Gauss} he then uses 
this result to prove that the number of lattice points (associated to
squares of side $a$), situated inside the figure, multiplied by the
area $a^2$ of the square, is asymptotic to the area of the figure if
$a$ tends to zero.

In fact, Gauss proves a result valid for general figures, though
without stating precise assumptions on the figure.
This, though not explicitly highlighted, gives 
$$|B(R) \cap \Z^2| = \pi R^2+O(R) $$
with an explicit
constant.

We remark that many modern sources attribute a slightly different
proof to Gauss, namely bounding the lattice points from above and
below in terms of the areas of two circles, one of radius
$R-1/\sqrt{2}$, and one of radius $R+1/\sqrt{2}$.
We have not been able to locate such a
proof in Gauss's writings.

\section{Heuristics for Hardy's conjecture}
\label{appendixB}
Let $N(R) := | B(R) \cap \Z^2|$ and define $E(R) := N(R) -
\pi R^{2}$. In \cite{hardy-average-order}, Hardy asserted that 
``\ldots and it is not unlikely that $E(R) = O(R^{1/2+\epsilon})$ for
all positive values of $\epsilon$, though this has never been
proved.'', and then showed that it is ``true on average'' in 
 that
$$
\frac{1}{x }\int_{0}^{x} |E(R)| dR = O(x^{1/2+\epsilon})
$$
as $x \to \infty$.

A folklore heuristic is that Hardy's conjecture would follow from
square root cancellation when summing over ``boundary terms'',
provided these terms are distributed as some sort of i.i.d. random
variables.

Given our setup we can give a probabilistic/dynamical heuristic that
exhibits an interesting ``perfect cancellation'' phenomena between the
curved and horizontal part of the boundary of a {\em quarter circle.}  The
number of lattice points in the first (half open) quarter circle,
i.e., $|\{ v \in \Z^2 : |v| \le R, \arg(v) \in [0, \pi/2) \}|$ then
equals $N(R)/4$. With $M(R)$ denoting the number of ``internal''
lattice points, i.e.,
$$ M(R) := |\{ (x,y) \in \Z^2 : x,y > 0, \arg((x,y)) \in [0, \pi/2)\}|
$$
we have $N(R)/4 = M(R) + R + O(1)$. Moreover, computing the area of
the quarter circle we have
$$
M(R) \cdot 1 + \sum_{\lambda \in \Lambda(R,(0,\pi/2))} A(\lambda)
=
\pi R^{2}/4
$$
which gives
$$
N(R)/4 = \pi R^{2}/4 - R + \sum_{\lambda \in \Lambda(R,(0,\pi/2))} A(\lambda).
$$
Since $|\Lambda(R,(0,\pi/2))| = 2R+O(1)$, we obtain 
$$
N(R) = \pi R^{2} + 4\sum_{\lambda \in \Lambda(R,(0,\pi/2))} \left(A(\lambda)-\frac{1}{2}\right)+O(1).
$$
Thus,
if the mean area is $1/2+o(1)$, i.e.,
$$
\frac{1}{\Lambda(R,(0,\pi/2))}
\sum_{\lambda \in \Lambda(R,(0,\pi/2))} A(\lambda)
=
1/2 + o(1)
$$
we find that 
$$
N(R) = \pi R^{2} + o(R)
$$
which gives a slight improvement on Gauss' error term $O(R)$.

We remark that $|\Lambda(R,(0,\pi/2)| \sim c\cdot R$ for $c \neq 2$
would imply that Gauss' error term $O(R)$ is {\em sharp}. In fact, any
improvement beyond Gauss crucially depends on a leading order
cancellation between two boundary terms: one from lattice points on
the horizontal axis, the other related to the areas of boundary
squares intersected with $B(R)$.

Finally, if $\{A(\lambda)-1/2\}_{\lambda \in \Lambda(R,(0,\pi/2)}$
``behaves as a collection of I.I.D random variables with mean zero'',
one would obtain
$N(R) = \pi R^{2} + O(R^{1/2+\epsilon})$ for any $\epsilon>0$.

\subsection{Cancellation via  trapezoidal integration}
A similar cancellation occurs when counting lattice
points by taking the integer part of $\sqrt{R^{2}-x^{2}}$; writing
$\rho(x) = 1/2 - \{ x \}$ where $\{x\} = x - [x]$ denotes the
fractional part of $x$, one finds a similar phenomena.
Namely, following Jameson \cite{jameson}, let
$f(x) = \sqrt{R^{2}-x^{2}} -x$ and put $K = [R/\sqrt{2}]$; then
$M(R)$, the number of integers in the half open quarter circle equals,
$$
2 \sum_{k=1}^{K}
[ f(k) ] + K + [R]=
2 \sum_{k=1}^{K} f(k) 
+
2 \sum_{k=1}^{K} \rho(f(k)) + [R]
$$
and using the trapezoidal rule we have
$$
\sum_{k=1}^{K} f(k) =
\frac{\pi R^{2}}{8} - R/2 + O(1)
$$
and obtain
$$
M(R) =
\pi R^{2}/4 + 
2 \sum_{k=1}^{K} \rho(f(k)). 
$$
Again there is essentially perfect cancellation between lattice points
on the $y$-coordinate axis (the term $[R]$ above) and the difference
between the $\sum_{k=1}^{K} f(k)$ and the integral
$\int_{0}^{K} f(t) \, dt$.

\end{document}